\documentclass[12pt]{amsart}
\usepackage{verbatim, latexsym, amssymb, amsmath,color, mathabx}
\usepackage{enumitem}
\usepackage{epsfig}
\usepackage{geometry}
\usepackage{hyperref}

\newtheorem{thm}{Theorem}[section]
\newtheorem{prop}[thm]{Proposition}
\newtheorem{lem}[thm]{Lemma}

\newtheorem{cor}[thm]{Corollary}
\newtheorem{rem}[thm]{Remark}

\DeclareMathOperator{\bR}{\mathbb{R}}

\DeclareMathOperator{\cG}{\mathcal{G}}
\DeclareMathOperator{\cL}{\mathcal{L}}
\DeclareMathOperator{\cP}{\mathcal{P}}
\DeclareMathOperator{\cR}{\mathcal{R}}
\DeclareMathOperator{\cS}{\mathcal{S}}

\DeclareMathOperator{\Lk}{\text{Lk}}
\DeclareMathOperator{\link}{\text{link}}

\begin{document}
\title[Right-handed flows]{Volume-preserving right-handed vector fields are conformally Reeb}
\author{Rohil Prasad} 
\begin{abstract} Right-handed and Reeb vector fields are two rich classes of vector fields on closed, oriented three-manifolds. Prior work of Dehornoy and Florio--Hryniewicz has produced many examples of Reeb vector fields which are right-handed. We prove a result in the other direction. We show that the closed two-form associated to a volume-preserving right-handed vector field is contact-type. This implies that any volume-preserving right-handed vector field is equal to a Reeb vector field after multiplication by a positive smooth function. Combining our result with theorems of Ghys and Taubes shows that any volume-preserving right-handed vector field has a global surface of section.  
\end{abstract}

\maketitle

\section{Introduction} 

\subsection{Statement of main results} The purpose of this note is to show the following theorem regarding volume-preserving right-handed vector fields.

\begin{thm}
\label{thm:mainTechnical} Let $X$ be a volume-preserving right-handed vector field on a closed, oriented rational homology three-sphere $M$ with volume form $\Omega$. Then the two-form $\omega = \Omega(X, -)$ is contact-type. 
\end{thm}

Theorem \ref{thm:mainTechnical} implies the titular theorem of this note, which states that a volume-preserving right-handed vector field is equal to a Reeb vector field after multiplication by a positive smooth function.

\begin{thm}
\label{thm:main} Any volume-preserving right-handed vector field $X$ on a closed, oriented rational homology three-sphere $M$ is conformally Reeb; there is a contact form $\lambda$ on $M$ with Reeb vector field $R$ and a positive smooth function $f > 0$ such that $R = fX$. 
\end{thm}

Right-handed vector fields are a special class of non-singular vector fields on closed, oriented rational homology three-spheres introduced by Ghys \cite{Ghys09}. Informally, a non-singular vector field $X$ on a closed, oriented rational homology three-sphere is right-handed if any pair of long embedded loops which are nearly tangent to $X$ link positively. One example is the vector field generating the \emph{Hopf fibration} on the three-sphere $S^3$, where any pair of orbits form a Hopf link with linking number $1$. Theorem \ref{thm:main} indicates that, for any right-handed vector field $X$, there is a Reeb vector field $R$ with the same simple periodic orbits as $X$.

Taubes' proof \cite{TaubesWeinstein} of the Weinstein conjecture in dimension three shows that any Reeb flow on a closed $3$-manifold has at least one simple periodic orbit; a quantitative refinement by Cristofaro-Gardiner--Hutchings \cite{CGH} shows that there are at least two. We deduce the following corollary from this discussion and Theorem \ref{thm:main}. 

\begin{cor}
\label{cor:orbits} Any volume-preserving right-handed vector field $X$ on a closed, oriented rational homology three-sphere $M$ has at least two simple periodic orbits. 
\end{cor}

Theorem \ref{thm:main} is sharp in the sense that there are volume-preserving right-handed vector fields which are not Reeb vector fields; right-handed vector fields remain right-handed under conformal change while the same is not true for Reeb vector fields. As a concrete example, consider the Reeb vector field $R$ on $S^3 \subset \mathbb{R}^4$ associated to the standard Liouville contact form. The vector field $R$ generates the Hopf fibration, so all of its orbits are periodic. Wadsley's theorem\footnote{The theorem is stated for geodesible vector fields, but all Reeb vector fields are geodesible.} \cite{Wadsley75} shows that for any Reeb vector field on a closed, connected, oriented three-dimensional manifold, all of whose orbits are periodic, the orbits share a common period. We fix a positive function $f$ on $S^3$ such that the vector field $X = fR$ has all of its orbits closed, but two periodic orbits have rationally independent minimal periods. Therefore, $X$ cannot be a Reeb vector field, but it is right-handed because it is a multiple of the right-handed vector field $R$ by a positive smooth function. It is also volume-preserving; if we write $\Omega$ for the volume form on $S^3$ then $X$ preserves the volume form $f^{-1}\Omega$. 

Right-handed vector fields are interesting from a dynamical perspective because they potentially admit an abundance of global surfaces of section. A \emph{global surface of section} for a vector field $X$ on a three-manifold $M$ is a compact, embedded, oriented\footnote{We follow the orientation conventions in \cite[Remark $1.4$]{FlorioHryniewicz21}. The orientation on a global surface of section $\Sigma$ is the one determined by the orientation of $M$ and the co-orientation of $\Sigma \setminus \partial\Sigma$ by $X$.} surface $\Sigma$ with oriented boundary $\partial\Sigma$ equal to a union of periodic orbits of $X$, such that $\Sigma \setminus \partial\Sigma$ is transverse to $X$ and at any point $p \in M$ the forward and backward orbits of $p$ under the flow of $X$ both intersect $\Sigma$. The first return map for the surface of section $\Sigma$ induces a diffeomorphism from $\Sigma$ to itself, effectively reducing the study of the three-dimensional flow of $X$ to the study of a two-dimensional surface diffeomorphism. 

This has been used to great effect, for example, by Hofer--Wysocki--Zehnder \cite{HWZ} and Cristofaro-Gardiner--Hutchings--Pomerleano \cite{CGHP} to show that many Reeb flows have two or infinitely many periodic orbits. The general strategy is to show that the Reeb flow has a genus zero surface of section $\Sigma$ with at least one boundary component (which relies on pseudoholomorphic curve techniques in both cases). Then the first return map of the Reeb flow induces an area-preserving diffeomorphism on $\Sigma$, and any periodic orbit of the Reeb flow either is a boundary component or induced by an periodic point in $\Sigma$. Then one can appeal to a theorem of Franks \cite{franks} showing that area-preserving maps of an annulus have zero or infinitely many periodic points, which given the above proves the desired result.
 
\begin{thm}\label{thm:ghys}
\cite{Ghys09} Let $X$ be a right-handed vector field on a closed, oriented rational homology three-sphere $M$. Then any collection of periodic orbits of $X$ bound a global surface of section. 
\end{thm}

We refer the reader to Dehornoy \cite{Dehornoy17} and Florio--Hryniewicz \cite{FlorioHryniewicz21} for results demonstrating that certain geodesic and Reeb flows, respectively, are right-handed. An alternate proof of Theorem \ref{thm:ghys} can also be found in \cite[Appendix $B$]{FlorioHryniewicz21}. Corollary \ref{cor:orbits} and Theorem \ref{thm:ghys} imply the unconditional existence of global surfaces of section for volume-preserving right-handed vector fields. 

\begin{thm} \label{thm:gss}
Any volume-preserving right-handed vector field $X$ on a closed, oriented rational homology three-sphere $M$ with volume form $\Omega$ admits a global surface of section.
\end{thm}

\subsection{Conventions and notation} For the rest of this note, we fix the following notation. Fix a closed, oriented rational homology three-sphere $M$ equipped with a volume form $\Omega$ of volume $1$. Fix a smooth, non-singular volume-preserving vector field $X$ with flow $\{\phi^t\}_{t \in \bR}$ and define the two-form $\omega = \Omega(X, -)$. An application of Cartan's formula to the identity $L_X\Omega = 0$ shows that $\omega$ is closed. We will at times abuse notation and regard the three-form $\Omega$ as an $X$-invariant probability measure. Fix also the notation $\cP(X)$ for the space of $X$-invariant Borel probability measures, and fix a Riemannian metric $g$ on $M$ such that $\|X\|_g \equiv 1$. We do not require that the volume form associated to $g$ is equal to $\Omega$. 

\subsection{Outline of proof} The strategy of the proof of Theorem \ref{thm:mainTechnical} is as follows. We will define, via a Hopf-type integral, a function $\Lk_\omega$ on the space of $X$-invariant probability measures. This is another version of the ``linking number'' of an invariant measure $\mu$ and the volume form $\Omega$ first introduced in \cite{Ghys09}. We will show that if $X$ is right-handed, then $\Lk_\omega(\mu)$ is positive for any $X$-invariant probability measure $\mu$. It is then immediate from McDuff's contact-type criterion \cite{McDuff87} that $\omega$ is contact-type as desired.

\subsection{Acknowledgements} I would like to thank my advisor, Helmut Hofer, as well as Julian Chaidez, Umberto Hryniewicz, and Georgios Kotsovolis for useful conversations regarding this work. This work is supported by the NSF under Award \#DGE-1656466. 

% A close reading of Taubes' work \cite{TaubesErgodic} on uniquely ergodic vector fields in dimension three shows that either $X$ has a periodic orbit or $\Lk_\omega$ has a zero. Since $X$ is right-handed, the latter cannot be true, so it must have a periodic orbit. 

\section{The Gauss linking form} \label{sec:gaussLinking}

Our main technical tool will be the \emph{Gauss linking form}, as introduced by Vogel \cite{Vogel03}. The Gauss linking form, which we denote by $L$, is a \emph{double form} on the product manifold $M \times M$. The \textbf{bundle of double forms} $$\pi_L^*(\Lambda^*(T^*M)) \otimes \pi_R^*(\Lambda^*(T^*M))$$ over $M \times M$ is the tensor product of the pullbacks of the bundle $\Lambda^*(T^*M)$ of differential forms by the projections
$$\pi_L,\,\pi_R: M \times M \to M.$$

Any differential form $\eta$ induces two distinct double forms $\eta^L$ and $\eta^R$, constructed by pulling back $\eta$ by $\pi_L$ and $\pi_R$ respectively. There are left and right exterior derivative operators $d_L$ and $d_R$ acting on double forms, as well as left and right Hodge star operators $\star_L$ and $\star_R$ induced by the Riemannian metric. The grading on $\Lambda^*(T^*M)$ induces a bigrading on the bundle of double forms. The operators $d_L$ and $d_R$ have bidegree $(1,0)$ and $(0,1)$.

Let $\cG$ denote the \emph{Green's form} of the Hodge Laplacian $\Delta$ associated to the Riemannian metric. This is an integrable double form, smooth outside the diagonal and satisfying the pointwise bound
$$\|\cG_{p,q}\| \lesssim \text{dist}(p,q)^{-1},$$
such that for any differential form $\beta$ the differential form
$$\eta_p = \int_{M} \cG_{p,q} \wedge \beta_q$$
satisfies the equation $\Delta \eta = \beta$. An explicit construction, using the Riemannian distance function to construct a parametrix, is found in \cite[Chapter V]{deRham}. The \emph{Gauss linking form} is defined by
$$\cL = \star_R d_R \cG.$$

It is immediate that $\cL$ is integrable and smooth outside of the diagonal in $M \times M$, satisfying the pointwise bound
\begin{equation} \label{eq:linkingFormBound} \|\cL_{p,q}\| \lesssim \text{dist}(p,q)^{-2}.\end{equation}

It will also be important to note the following behavior of $\cL$ with respect to the volume measure $\Omega$, which is also mentioned in \cite{KotschickVogel}. Write $\|\cL\|_R$ for the norm of $\cL$ in the right factor, which in terms of local product coordinates $(p,q)$ on $M \times M$ is equal to a differential form in the $p$-coordinates which may have coefficients depending on both $p$ and $q$. Then for each $p \in M$ the differential form
$$\int_M \|\cL_{p,q}\|_R \Omega(q)$$
exists and varies continuously in $p$. Next, we note the following technical property of the Gauss linking form $\cL$, proved in \cite{Vogel03}.

\begin{lem}
\label{lem:linkingFormIdentity} \cite[Lemma $2$]{Vogel03} For any smooth one-form $\eta$ on $M$ there is a smooth function $h$ on $M$ such that for any $p \in M$,
$$\eta(p) - \int_M (d\eta)^R_q \wedge \cL_{p,q} = dh(p).$$
\end{lem}

Lemma \ref{lem:linkingFormIdentity} was used by Vogel to show that $\cL$ can be used to compute the linking number of two loops in $M$. It can be thought of as saying that $\cL$ is an ``inverse'' to the exterior derivative operator up to an error given by an exact form. 

\begin{lem}
\label{lem:gaussLinking} \cite[Theorem $3$]{Vogel03} Let $\gamma_1$ and $\gamma_2$ be a pair of disjoint, oriented embedded loops in $M$. Then the linking number $\text{link}(\gamma_1, \gamma_2)$ is equal to the integral
$$\int_{\gamma_1 \times \gamma_2} \cL.$$
\end{lem}

The final property of the Gauss linking form we need is the boundedness of its integral on pairs of short geodesics.

\begin{lem}
\label{lem:gaussLinkingGeodesics} Fix $r_0 \leq \text{injrad}(g)/100$. Then there is a constant $C = C(g) > 0$ depending on the Riemannian metric such that for any pair of embedded geodesics $\gamma_1$ and $\gamma_2$ of length less than or equal to $r_0$, intersecting at most once, 
$$|\int_{\gamma_1 \times \gamma_2} \cL| \leq C.$$
\end{lem}

Lemma \ref{lem:gaussLinkingGeodesics} is stated less formally in the proof of \cite[Theorem $5$]{Vogel03}, as well as in \cite{ArnoldKhesin}. It follows from an explicit computation in the case where the Riemannian manifold $(M, g)$ is $\bR^3$ with the Euclidean metric; the case of a general Riemannian metric on a closed $3$-manifold can be reduced to the Euclidean case by taking geodesic normal coordinates. Our main application, Proposition \ref{prop:LkOmegaRightHanded} below, will as a consequence of making some convenient assumptions on the Riemannian metric $g$ require only the computation in the Euclidean case. 

To illustrate the idea behind this computation for the reader, we discuss it in a simple case. When $M = \bR^3$ and $g$ is the Euclidean metric, the Gauss linking form is the double form
\begin{equation} \label{eq:euclideanLinkingForm} \cL_{p,q}(V, W) = \frac{1}{4\pi}\frac{\langle V, W \times (p - q)\rangle}{\|V - W\|^3}.\end{equation}

The ``$\times$'' symbol denotes in this case the cross product of vectors in $\bR^3$. Let $\gamma_1, \gamma_2: \bR \to \bR^3$ be two straight lines intersecting once. It follows that for any $s$ and $t$, the vector $\gamma_1(s) - \gamma_2(t)$ is a linear combination of $\dot{\gamma}_1(s)$ and $\dot{\gamma}_2(t)$, from which it follows that 
$$\langle \dot{\gamma}_1(s), \dot{\gamma}_2(t) \times (\gamma_1(s) - \gamma_2(t)) \rangle \equiv 0$$
which using (\ref{eq:euclideanLinkingForm}) implies that the integral of $\cL$ over the product of the two lines vanishes. 

\section{The linking of an invariant measure with the volume}

\subsection{The function $\Lk_\omega$} 

For every $X$-invariant probability measure $\mu \in \cP(X)$, choose a primitive $\nu$ of $\omega$ and define the quantity
$$\Lk_\omega(\mu) = \int_M \nu(X) d\mu.$$

\begin{lem}
For any $\mu \in \cP(X)$, the quantity $\text{Lk}_\omega(\mu)$ does not depend on the choice of primitive $\nu$ of $\omega$.
\end{lem}

\begin{proof}
Since $M$ is a rational homology three-sphere, any two primitives $\nu_1$ and $\nu_2$ of $\omega$ differ by an exact one-form $df$ for $f \in C^\infty(M)$. Let $\mu$ be an $X$-invariant probability measure. Since it is $X$-invariant, we find for any smooth function $g: M \to \mathbb{R}$ that
$$\int_M dg(X) d\mu = 0.$$

We use this to show that the two versions of $\Lk_\omega$ for $\nu_1$ and $\nu_2$ coincide:
\begin{equation*}
    \int_M (\nu_1(X) - \nu_2(X)) d\mu = \int_M df(X) d\mu = 0.
\end{equation*}
\end{proof}

\subsection{Right-handed vector fields} \label{subsec:rightHanded} We now describe Ghys' notion of right-handedness of a non-singular vector field on a rational homology three-sphere. We use the definition given in \cite[Section $2.2$]{FlorioHryniewicz21}. 

Let $Y$ be a non-singular smooth vector field on $M$. Denote the flow of $Y$ by $\{\psi^t\}_{t \in \bR}$. Denote the set of recurrent points of the flow of $Y$ by $\cR(Y)$. There is an associated measurable subset
$$R(Y) = \{(p,q) \in \cR(Y) \times \cR(Y)\,|\,\psi^{\bR}(p) \cap \psi^{\bR}(q) = \emptyset\}$$
of $M \times M$. 

\subsubsection{Positive linking of ergodic invariant measures} Fix any two ergodic measures $\mu_1$ and $\mu_2$ in $\cP(Y)$. It follows by the ergodicity assumption and the definition of $R(Y)$ that one of the following two conditions hold. 

\begin{comment}
We sketch here a proof that at least one of Condition $1$ and Condition $2$ holds. Suppose that $(\mu_1 \times \mu_2)(R) < 1$. Let $E \subset \mathcal{R}(Y)$ be the invariant set of $p \in M$ such that $\mathcal{R}(Y) \setminus \psi^{\bR}(p)$ has zero $\mu_2$-measure. Then our assumption that $(\mu_1 \times \mu_2)(R) < 1$ implies that $E$ has positive measure, and since $\mu_1$ is ergodic this implies $E$ has full $\mu_1$-measure. 

Note that $\mathcal{R}(Y)$ has full measure by Poincar\'e recurrence, so we conclude from this that $\mu_2(\psi^{\bR}(p)) = 1$ for each $p \in E$. This implies by invariance and countable subadditivity of $\mu_2$ that $\psi^{\bR}(p)$ is a closed orbit for any $p \in E$ and $\mu_2$ is supported within $\psi^{\bR}(p)$. It follows from this that for any $p$ and $q$ in $E$, $\psi^{\bR}(p) = \psi^{\bR}(q)$. Call this common orbit $\gamma$. We conclude that $E$ is equal to $\gamma$ and so $\mu_1$ and $\mu_2$ are both supported in $\gamma$.
\end{comment}

\textbf{Condition 1:} $(\mu_1 \times \mu_2)(R) = 1$. We define in this case the notion of \textbf{positive linking} of $\mu_1$ and $\mu_2$. Fix $(p, q) \in R(X)$ and let $\cS(p, q)$ denote the set of ordered pairs of monotonically increasing sequences $(\{S_n\}, \{T_n\})$ such that $S_n, T_n \to \infty$, $\psi^{S_n}(p) \to p$ and $\psi^{T_n}(q) \to q$. For $n \gg 1$ let $\alpha_n = \sigma(\psi^{S_n}(p), p)$ and $\beta_n = \sigma(\psi^{T_n}(q), q)$ be the shortest geodesic arcs from $\psi^{S_n}(p)$ to $p$ and from $\psi^{T_n}(q)$ to $q$, respectively. For each $n$, fix $C^1$-small perturbations $\hat{\alpha}_n$, $\hat{\beta}_n$ of $\alpha_n$ and $\beta_n$, respectively, so that the loops 
$$k(S_n, p) = \psi^{[0,S_n]}(p) \cup \hat{\alpha}_n$$ 
and 
$$k(T_n, q) = \psi^{[0,T_n]}(q) \cup \hat{\beta}_n$$ 
do not intersect. Define
$$\link_-(\psi^{[0,S_n]}(p), \psi^{[0,T_n]}(q)) = \liminf_{\hat{\alpha}_n \xrightarrow{C^1} \alpha_n, \hat{\beta}_n\xrightarrow{C^1}\beta_n} \link(k(S_n,p), k(T_n,q))$$
and
$$\ell(p,q) = \inf_{(\{S_n\},\{T_n\}) \in \cS(p,q)} \liminf_{n \to \infty} \frac{1}{S_n T_n}\link_-(\psi^{[0,S_n]}(p), \psi^{[0,T_n]}(q)).$$

We say that $\mu_1$ and $\mu_2$ are \textbf{positively linked} if there is a subset $E \subset R(Y)$ of full $(\mu_1 \times \mu_2)$-measure such that $\ell(p,q) > 0$ for every $(p, q) \in E$. 

\textbf{Condition 2:} $(\mu_1 \times \mu_2)(R) = 0$ and $\text{supp}(\mu_1) \cup \text{supp}(\mu_2) = \gamma$ for some periodic orbit $\gamma$. We say $\mu_1$ and $\mu_2$ are \textbf{positively linked} if the periodic orbit $\gamma$ has positive rotation number in a Seifert framing. 

\subsubsection{Definition of right-handedness} Given the definition of positive linking of pairs of ergodic measures above, we say that the vector field $Y$ is \textbf{right-handed} if any pair of ergodic measures $\mu_1$ and $\mu_2$ is positively linked. 

\subsection{$\Lk_\omega$ is positive for right-handed vector fields} The aim of this subsection is to show the following proposition.

\begin{prop} \label{prop:LkOmegaRightHanded}
Let $\mu \in \cP(X)$ be any $X$-invariant probability measure. If the vector field $X$ is right-handed then $\Lk_\omega(\mu) > 0$. 
\end{prop}

We use Lemma \ref{lem:linkingFormIdentity} to compute $\Lk_\omega$ in terms of the Gauss linking form. Let $\cL$ be the Gauss linking form from Section \ref{sec:gaussLinking}. For any pair of vector fields $Y_1$ and $Y_2$, we will use the notation $\cL(Y_1, Y_2)$ to denote the contraction of the bidegree $(1,1)$ part of the double form $\cL$ with $Y_1$ and $Y_2$. Lemma \ref{lem:lkOmegaComputation} is essentially contained in the work of Kotschick--Vogel \cite{KotschickVogel}, using the correspondence between invariant measures $\mu \in \cP(X)$ and Ruelle-Sullivan cycles for $X$. 

\begin{lem}
\label{lem:lkOmegaComputation}
For any $X$-invariant probability measure $\mu$, $\cL(X, X)$ is integrable with respect to the product measure $d\mu \times \Omega$ and 
$$\Lk_\omega(\mu) = \int_{M \times M} \cL_{p,q}(X, X) (d\mu \times \Omega)(p,q).$$

%\int_M \big(\int_M \cL_{p,q}(X, X)\Omega(q)\big) d\mu(p).$$
\end{lem}

\begin{proof}
Observe that for any exact one-form $dh$, the fact that $\mu$ is an invariant probability measure implies that
$$\int_M dh(X) d\mu = 0.$$

Then the lemma follows from Lemma \ref{lem:linkingFormIdentity}: 
\begin{align*}
    \Lk_\omega(\mu) &= \int_M (\iota_X \nu)_p d\mu(p) \\
    &= \int_M \iota_X\big(\int_M \omega_q \wedge \cL_{p,q}) d\mu(p) \\
    &= \int_M \iota_X \big(\int_M \cL_{p,q}(-, X) \Omega(q)\big) d\mu(p) \\
    &= \int_M \big(\int_M \cL_{p,q}(X, X)\Omega(q)\big) d\mu(p).
\end{align*}

For any fixed $p \in M$, the function $\cL_{p,q}(X, X)$ is integrable with respect to the measure $\Omega(q)$. From what was said in Section \ref{sec:gaussLinking}, the function
$$\int_M |\cL_{p,q}(X, X)| \Omega(q)$$
varies continuously in $p$. The lemma then follows from Fubini's theorem and the prior computation. 
\end{proof}

We are now equipped to prove Proposition \ref{prop:LkOmegaRightHanded}. This will use the above computation, Tempelman's multiparameter pointwise ergodic theorem (\cite{Tempelman72}, \cite[Chapter $6$]{TempelmanBook}), and the fact that the integral of the linking form computes linking numbers of loops. For the purposes of computation, we will assume that the Riemannian metric $g$ takes on the following form. Fix some positive constant $r_0 \ll 1$. Write $(\overline{B}_{r_0}(0), g_{\text{Euc}})$ for the closed Euclidean ball in $\mathbb{R}^3$ of radius $r_0$. Then we assume that there are isometric embeddings
$$\iota_1,\,\iota_2: (\overline{B}_{r_0}(0), g_{\text{Euc}}) \hookrightarrow (M, g)$$
with disjoint images such that
$$(\iota_i)_*(\partial_z) = X$$
for $i \in \{1,2\}$, where we write the coordinates on $\mathbb{R}^3$ as $(x, y, z)$. 

\begin{proof}[Proof of Proposition \ref{prop:LkOmegaRightHanded}]
Suppose $X$ is right-handed. We will assume for the sake of contradiction that $\Lk_\omega(\mu) \leq 0$ for some $\mu \in \cP(X)$. The proof then proceeds in $6$ steps.

\textbf{Step 1:} This step reduces the proposition to the case where $\mu$ is ergodic. The ergodic decomposition theorem shows that if 
$$\Lk_\omega(\mu) = \int_M \nu(X) d\mu \leq 0,$$ for some $\mu \in \cP(X)$ then there must be some ergodic probability invariant measure $\mu_*$ such that $\Lk_\omega(\mu_*) \leq 0$. Therefore, we assume without loss of generality in the following steps that there is an \emph{ergodic} measure $\mu \in \cP(X)$ such that $\Lk_\omega(\mu) \leq 0$ and derive a contradiction, which will prove the proposition. We will split up the proof into two cases and derive a contradiction in each case. The first case, which is the subject of Steps $2$ through $5$, assumes that $\mu$ does not have support in a periodic orbit of $X$. The second case, which is the subject of Step $6$, assumes that $\mu$ is supported in a periodic orbit of $X$. 

\textbf{Step 2:} Our assumptions and the computation from Lemma \ref{lem:lkOmegaComputation} imply the inequality
$$\int_{M \times M} \cL_{p,q}(X, X) (d\mu \times \Omega)(p, q) \leq 0.$$

It is also important to note that $\cL_{p,q}(X, X) \in L^1(M \times M; \mu \times \Omega)$. 

\textbf{Step 3:} Another application of the ergodic decomposition theorem to the inequality from the previous step produces an ergodic invariant probability measure $\eta \in \cP(X)$ such that $\cL_{p,q}(X, X) \in L^1(M \times M, \mu \times \eta)$ and 
$$\int_{M \times M} \cL_{p,q}(X, X) (d\mu \times d\eta)(p, q) \leq 0.$$

Write $R = R(X)$ for the set introduced in Section \ref{subsec:rightHanded}. Since we are assuming that $\mu$ does not have support contained in a periodic orbit, it follows that $(\mu \times \eta)(R) = 1$. Since $X$ is right-handed the measures $\mu$ and $\eta$ are ``positively linked'' in the sense described in Condition $1$ in Section \ref{subsec:rightHanded}. Tempelman's multiparameter pointwise ergodic theorem (\cite{Tempelman72}, \cite[Corollary $6.3.3$]{TempelmanBook}) tells us that the pointwise limit
$$\lim_{S, T \nearrow \infty} \frac{1}{ST} \int_{\phi^{[0,S]}(p) \times \phi^{[0,T]}(q)} \cL = \int_{M \times M} \cL_{p,q}(X, X)(d\mu \times d\eta)(p, q)$$
exists almost everywhere. Here the notation ``$\nearrow$'' indicates that $S$ and $T$ are taken to be monotonically increasing. 

The prior inequality tells us that there is a Borel set $E \subset M \times M$ such that $(\mu \times \eta)(E) = 1$ and for any $(p, q) \in E$, the limit
$$\lim_{S, T \nearrow \infty} \frac{1}{ST} \int_{\phi^{[0,S]}(p) \times \phi^{[0,T]}(q)} \cL$$
exists and is non-positive. 

\textbf{Step 4:} In Steps $4$ and $5$, we use the inequality from Step $3$ and the fact that the ergodic measures $\mu$ and $\eta$ are positively linked to derive a contradiction. Recall the quantity $\ell(p,q)$ defined in Section \ref{subsec:rightHanded}. Since $\mu$ and $\eta$ are positively linked, we can choose a pair of points $(p, q) \in E \cap R$ such that $\ell(p,q) > 0$. Moreover, we can assume without loss of generality that, after possibly shrinking $r_0$ slightly, 
$$p = \iota_1(0)$$
and
$$q = \iota_2(0)$$
where $\iota_1$ and $\iota_2$ are the isometric embeddings of Euclidean balls $(\overline{B}_{r_0}(0), g_{\text{Euc}})$ fixed prior to the beginning of the proof of the proposition. 

It follows from the prior step that for any pair $(\{S_n\}, \{T_n\}) \in \mathcal{S}(p,q)$, we have 
\begin{equation} \label{eq:lkOmegaRightHanded1} \lim_{n \to \infty} \frac{1}{S_n T_n} \int_{\phi^{[0,S_n]}(p) \times \phi^{[0,T_n]}(q)} \cL \leq 0. \end{equation}

For any $n$, let $\alpha_n$ and $\beta_n$ be the shortest geodesic arcs from $\phi^{S_n}(p)$ to $p$ and from $\phi^{T_n}(q)$ to $q$, respectively. Choose for any $n$ small $C^1$ perturbations $\hat{\alpha}_n$ of $\alpha_n$ and $\hat{\beta}_n$ of $\beta_n$ such that the loops $k(S_n, p)$ and $k(T_n, q)$ given by closing up the flow lines $\phi^{[0,S_n]}(p)$ and $\phi^{[0,T_n]}(q)$ by $\hat{\alpha}_n$ and $\hat{\beta}_n$ are disjoint. Then using the fact that $\ell(p, q) > 0$, then it follows for any $n$ there is a constant $\epsilon_n = \epsilon_n(p, q, \{S_n\}, \{T_n\}) > 0$ depending on $n$, the pair $(p, q)$, and the pair $(\{S_n\}, \{T_n\})$ such that if 
$$\text{dist}_{C^1}(\alpha_n, \hat{\alpha}_n) + \text{dist}_{C^1}(\beta_n, \hat{\beta}_n) \leq \epsilon_n$$
for every $n$, we find 
$$\liminf_{n \to \infty} \frac{1}{S_n T_n}\text{link}(k(S_n, p), k(T_n, q)) > 0.$$

Observe by Lemma \ref{lem:gaussLinking} that for any $n$,
$$\text{link}(k(S_n, p), k(T_n, q)) = \int_{k(S_n, p) \times k(T_n, q)} \cL.$$

We conclude that 
\begin{equation} \label{eq:lkOmegaRightHanded2} \liminf_{n \to \infty} \frac{1}{S_n T_n}\int_{k(S_n, p) \times k(T_n, q)} \cL > 0\end{equation}
where we continue to assume that the $C^1$ perturbations are taken to be of size at most $\epsilon_n$ for any $n$. 

Subtracting (\ref{eq:lkOmegaRightHanded1}) from (\ref{eq:lkOmegaRightHanded2}) yields the inequality
\begin{equation}
    \label{eq:lkOmegaRightHanded3} \liminf_{n \to \infty} \frac{1}{S_n T_n} \big(\int_{\phi^{[0,S_n]}(p) \times \hat{\beta}_n} + \int_{\hat{\alpha}_n \times \phi^{[0,T_n]}(q)} + \int_{\hat{\alpha}_n \times \hat{\beta}_n}\big) \cL > 0.
\end{equation}

We will derive a contradiction by showing that, for some suitable choice of pair $(\{S_n\}, \{T_n\}) \in \cS(p,q)$ and subsequent choice of approximations $\hat{\alpha}_n$ and $\hat{\beta}_n$, there is some uniform constant $C \geq 1$ such that
\begin{gather}
    \label{ineq:contradiction1} |\int_{\phi^{[0,S_n]}(p) \times \hat{\beta}_n} \cL| \leq CS_n, \\
    \label{ineq:contradiction2} |\int_{\hat{\alpha}_n \times \phi^{[0,T_n]}(q)} \cL| \leq CT_n, \\
    \label{ineq:contradiction3} |\int_{\hat{\alpha}_n \times \hat{\beta}_n} \cL| \leq C.
\end{gather}

The bounds (\ref{ineq:contradiction1}--\ref{ineq:contradiction3}) are incompatible with the inequality (\ref{eq:lkOmegaRightHanded3}), which leads to a contradiction and therefore a proof of the proposition in the case where the measure $\mu$ is ergodic with support not contained in a periodic orbit. Note that the bound (\ref{ineq:contradiction3}) is straightforward using the pointwise bound (\ref{eq:linkingFormBound}) for the linking form, since by definition for sufficiently large $n$ the distance between $\hat{\alpha}_n$ and $\hat{\beta}_n$ is bounded below by $\text{dist}(p,q)/2$. 

\textbf{Step 5:} This step proves the inequality (\ref{ineq:contradiction1}). The argument for the inequality (\ref{ineq:contradiction2}) is identical after cosmetic changes, in particular replacing all uses of the embedding $\iota_2$ with the embedding $\iota_1$.  

Recall from the choice of $(p,q)$ in Step $4$ that there are isometric embeddings
$$\iota_1,\,\iota_2: (B_{r_0}(0), g_{\text{Euc}}) \hookrightarrow (M, g)$$
such that
\begin{itemize}
    \item $\iota_1(0) = p$ and $\iota_2(0) = q$.
    \item $(\iota_i)_*(\partial_z) = X$ for $i \in \{1,2\}$. 
\end{itemize}

For any positive $r \leq r_0$, we write $D_{r} = \{ (x, y, 0) \in B_{r_0}(0)\,|\,|x|^2 + |y|^2 < r^2\}$ for the open disk of radius $r$ in the $xy$-plane. We now fix the pair $(\{S_n\}, \{T_n\}) \in \cS(p,q)$ to be such that for any $n$, the point $\phi^{T_n}(q)$ lies in the embedded disk $\iota_2(D_{r_0/4} \times \{0\})$. For every $n$, write $(x_n, y_n, 0)$ for the unique point in $D_{r_0/4}$ such that
$$\iota_2(x_n, y_n, 0) = \phi^{T_n}(q).$$

It follows that the unique length-minimizing geodesic $\beta_n$ from $\phi^{T_n}(q)$ to $q$ is the composition of $\iota_2$ with the radial path
$$t \mapsto ((1-t)x_n, (1-t)y_n, 0)$$
for $t \in [0,1]$. Let $\gamma_n$ be any segment of the flow-line $\phi^{[0,S_n]}(p)$ of length $2r_0$. If $\gamma_n$ is disjoint from the image of the embedding $\iota_2$, it follows from the bound (\ref{eq:linkingFormBound}) that the integral of the linking form on $\gamma_n \times \beta_n$ is uniformly bounded in $n$:
\begin{equation} \label{ineq:outsideBound} |\int_{\gamma_n \times \beta_n} \cL| \lesssim 1. \end{equation}

We now only need to bound the integral of $\gamma_n \times \beta_n$ when $\gamma_n$ is not disjoint from the image of $\iota_2$. To do so, it suffices to consider the case where $\gamma_n$ is the composition of $\iota_2$ with the curve
$$t \mapsto (x, y, t)$$
for $t \in (-r_0, r_0)$ and some fixed $(x, y, 0) \in D_{r_0/2}$. This is because, in the coordinates given by $\iota_2$, the flow-lines of $X$ are flow-lines of the coordinate vector field $\partial_z$.

Fix any $(x, y, 0)$ in $D_{r_0/2}$ and set $\gamma_{x,y}$ to be the composition of $\iota_2$ with the curve
$$t \mapsto (x, y, t)$$
for $t \in (-r_0, r_0)$. Since the metric is Euclidean in the image of the embedding $\iota_2$, both $\gamma_{x,y}$ and $\beta_n$ are short geodesics. We also observe that $\gamma_{x,y}$ and $\beta_n$ intersect at most once, at a right angle. It then follows from Lemma \ref{lem:gaussLinkingGeodesics} that the integral of $\cL$ over $\gamma_{x,y} \times \beta_n$ is bounded independently of the choice of $(x, y)$ or $n$:
\begin{equation} \label{ineq:flowBoxBound} |\int_{\gamma_{x,y} \times \beta_n} \cL| \lesssim 1. \end{equation}

The inequalities (\ref{ineq:outsideBound}) and (\ref{ineq:flowBoxBound}) suffice to show the inequality (\ref{ineq:contradiction1}). This is because the flow-line $\phi^{[0,S_n]}$ is covered by $\lesssim S_n$ segments of size $2r_0$. We split up the integral on the left-hand side of (\ref{ineq:contradiction1}) along this cover and deduce the desired inequality by application of (\ref{ineq:outsideBound}) or (\ref{ineq:flowBoxBound}) to the integral corresponding to each segment. 

\textbf{Step 6:} We now address the case where $\mu$ is ergodic and supported in a periodic orbit $\gamma$ of the vector field $X$ of minimal period $S > 0$. It follows that 
$$\Lk_\omega(\mu) = \frac{1}{S}\int_\gamma \nu$$
and, by our assumption, 
$$\int_\gamma \nu \leq 0$$
where $\nu$ is any primitive of $\omega$. Theorem \ref{thm:ghys} implies that $\gamma$ bounds a global surface of section $\Sigma$. The orientation on $\Sigma$ inducing the orientation by $X$ on $\gamma$ coincides with the orientation induced by the ambient orientation on $M$ and the co-orientation of $\Sigma \setminus \gamma$ by $X$, see the convention in \cite[Remark $1.4$]{FlorioHryniewicz21}. Then $\Sigma$ is positively transverse to $X$ at any point not on its boundary, and it follows that the two-form $\omega$ is pointwise positive on the tangent plane at any non-boundary point of $\Sigma$. We then conclude using Stokes theorem the inequality
$$0 < \int_\Sigma \omega = \int_\gamma \nu \leq 0$$
and therefore arrive at a contradiction. 
\end{proof}

\begin{rem}
A more elegant proof of Proposition \ref{prop:LkOmegaRightHanded} can be derived from the results of Ghys' original paper \cite{Ghys09} on right-handed vector fields. Ghys demonstrates the existence of a ``universal linking form'' $\overline{\cL}$, a double form such that
\begin{itemize}
    \item $\overline{\cL} = \cL + d_L d_R f$, where $f$ is some smooth function on $M \times M$. 
    \item The function $\overline{\cL}(X, X)$ is smooth and pointwise positive on the complement of the diagonal in $M \times M$.
\end{itemize} 

Then Proposition \ref{prop:LkOmegaRightHanded} is immediate from the properties of $\overline{\cL}$ and the computation in Lemma \ref{lem:lkOmegaComputation}. We present a different proof above to avoid delving into the details of Ghys' linking form for invariant measures. We also note that, following the approach of Kotschick--Vogel \cite{KotschickVogel}, one should be able to show that the linking number $\Lk_\omega(\mu)$ is equal to $\Lk_{\text{Ghys}}(\mu, \Omega)$, where we use $\Lk_{\text{Ghys}}$ to denote Ghys' linking form. This would provide a third proof of Proposition \ref{prop:LkOmegaRightHanded}, since by definition $\Lk_{\text{Ghys}}(\mu, \Omega) > 0$ when the vector field $X$ is right-handed. 
\end{rem}

\section{Proof of Theorems \ref{thm:mainTechnical} and \ref{thm:main}}

We will use McDuff's contact-type criterion \cite{McDuff87} to deduce Theorem \ref{thm:mainTechnical}, and as a consequence Theorem \ref{thm:main}. The two-form $\omega$ on $M$ is said to be \textbf{contact-type} if there is a primitive $\nu$ such that $\nu \wedge \omega > 0$. The contact-type criterion gives a condition on the structure boundaries of $X$ for $\omega$ to be contact-type. A \textbf{structure boundary}\footnote{These are typically called \textbf{structure cycles}, but since $M$ is a rational homology $3$-sphere every structure cycle is a structure boundary.} is a $1$-dimensional current $X \otimes \mu$, where $\mu$ is an $X$-invariant measure, defined by
$$(X \otimes \mu)(\alpha) = \int_M \alpha(X) d\mu$$
where $\alpha$ is any smooth $1$-forms. Then McDuff's contact-type criterion states that $\omega$ is contact-type if and only if there is a primitive $\nu$ of $\omega$ such that for any structure boundary $X \otimes \mu$, we have
$$(X \otimes \mu)(\nu) \neq 0.$$

Observe that since $\mu$ is an $X$-invariant measure, this criterion does not depend on the choice of primitive $\nu$. Indeed by definition we have
$$(X \otimes \mu)(\nu) = \Lk_\omega(\mu).$$

Therefore, McDuff's contact-type criterion in our setting can be phrased as the following theorem.

\begin{thm}\label{thm:mcduff}
\cite{McDuff87} Let $X$ be a non-singular volume-preserving vector field on a closed, oriented rational homology three-sphere $M$ with volume form $\Omega$. Then the two-form $\omega = \Omega(X, -)$ is contact-type if and only if $\Lk_\omega(\mu) \neq 0$ for any $X$-invariant probability measure $\mu$.
\end{thm}

% We observe in the context of Theorem \ref{thm:mcduff} that $\Lk_\omega(\mu) \neq 0$ for any $X$-invariant probability measure if and only if the numbers $\Lk_\omega(\mu)$ have the same \emph{sign} for every $X$-invariant probability measure. If this latter condition did not hold, there would exist a convex combination of two invariant measures with linking numbers of differing signs such that $\Lk_\omega(\mu) = 0$.

We now finish the proofs of Theorems \ref{thm:mainTechnical} and \ref{thm:main}. 

\begin{proof}[Proof of Theorem \ref{thm:mainTechnical}]
Combine Proposition \ref{prop:LkOmegaRightHanded} and Theorem \ref{thm:mcduff}. 
\end{proof}

\begin{proof}[Proof of Theorem \ref{thm:main}]
Theorem \ref{thm:mainTechnical} shows that $\omega$ is contact-type. Therefore, it admits a primitive $\nu$ such that $\nu \wedge \omega > 0$. The Reeb vector field $R$ of $\nu$ is the \emph{unique} vector field such that 
$$\nu(R) \equiv 1 \qquad \qquad \omega(R, -) \equiv 0.$$

Observe that since $\nu \wedge \omega > 0$, the smooth function $\nu(X)$ is everywhere positive. If we set $f = \nu(X)^{-1}$, we compute
$$\nu(fX) \equiv 1 \qquad \qquad \omega(fX, -) \equiv 0.$$

It follows that $R = fX$ as desired. 
\end{proof}

\bibliographystyle{alpha}
\bibliography{main}

\begin{thebibliography}{CGHP19}

\bibitem[AK21]{ArnoldKhesin}
Vladimir~I. Arnold and Boris~A. Khesin.
\newblock {\em Topological methods in hydrodynamics}, volume 125 of {\em
  Applied Mathematical Sciences}.
\newblock Springer, Cham, [2021] \copyright 2021.
\newblock Second edition [of 1612569].

\bibitem[CGH16]{CGH}
Daniel Cristofaro-Gardiner and Michael Hutchings.
\newblock From one {R}eeb orbit to two.
\newblock {\em J. Differential Geom.}, 102(1):25--36, 2016.

\bibitem[CGHP19]{CGHP}
Dan Cristofaro-Gardiner, Michael Hutchings, and Daniel Pomerleano.
\newblock Torsion contact forms in three dimensions have two or infinitely many
  {R}eeb orbits.
\newblock {\em Geom. Topol.}, 23(7):3601--3645, 2019.

\bibitem[Deh17]{Dehornoy17}
Pierre Dehornoy.
\newblock Which geodesic flows are left-handed?
\newblock {\em Groups Geom. Dyn.}, 11(4):1347--1376, 2017.

\bibitem[dR84]{deRham}
Georges de~Rham.
\newblock {\em Differentiable manifolds}, volume 266 of {\em Grundlehren der
  mathematischen Wissenschaften [Fundamental Principles of Mathematical
  Sciences]}.
\newblock Springer-Verlag, Berlin, 1984.
\newblock Forms, currents, harmonic forms, Translated from the French by F. R.
  Smith, With an introduction by S. S. Chern.

\bibitem[FH21]{FlorioHryniewicz21}
Anna Florio and Umberto Hryniewicz.
\newblock Quantitative conditions for right-handedness.
\newblock {\em arXiv preprint arXiv:2106.12512}, 2021.

\bibitem[Fra96]{franks}
John Franks.
\newblock Area preserving homeomorphisms of open surfaces of genus zero.
\newblock {\em New York J. Math.}, 2:1--19, electronic, 1996.

\bibitem[Ghy09]{Ghys09}
\'{E}tienne Ghys.
\newblock Right-handed vector fields \& the {L}orenz attractor.
\newblock {\em Jpn. J. Math.}, 4(1):47--61, 2009.

\bibitem[HWZ03]{HWZ}
H.~Hofer, K.~Wysocki, and E.~Zehnder.
\newblock Finite energy foliations of tight three-spheres and {H}amiltonian
  dynamics.
\newblock {\em Ann. of Math. (2)}, 157(1):125--255, 2003.

\bibitem[KV03]{KotschickVogel}
D.~Kotschick and T.~Vogel.
\newblock Linking numbers of measured foliations.
\newblock {\em Ergodic Theory Dynam. Systems}, 23(2):541--558, 2003.

\bibitem[McD87]{McDuff87}
Dusa McDuff.
\newblock Applications of convex integration to symplectic and contact
  geometry.
\newblock {\em Ann. Inst. Fourier (Grenoble)}, 37(1):107--133, 1987.

\bibitem[Tau07]{TaubesWeinstein}
Clifford~Henry Taubes.
\newblock The {S}eiberg-{W}itten equations and the {W}einstein conjecture.
\newblock {\em Geom. Topol.}, 11:2117--2202, 2007.

\bibitem[Tem72]{Tempelman72}
A.~A. Tempel'man.
\newblock Ergodic theorems for general dynamical systems.
\newblock {\em Trudy Moskov. Mat. Ob\v{s}\v{c}.}, 26:95--132, 1972.

\bibitem[Tem92]{TempelmanBook}
Arkady Tempelman.
\newblock {\em Ergodic theorems for group actions: Information and
  thermodynamical aspects}.
\newblock Kluwer, 1992.

\bibitem[Vog03]{Vogel03}
Thomas Vogel.
\newblock On the asymptotic linking number.
\newblock {\em Proc. Amer. Math. Soc.}, 131(7):2289--2297, 2003.

\bibitem[Wad75]{Wadsley75}
A.~W. Wadsley.
\newblock Geodesic foliations by circles.
\newblock {\em J. Differential Geometry}, 10(4):541--549, 1975.

\end{thebibliography}

\end{document}